\documentclass{article}

\usepackage[english]{babel}
\usepackage{amsmath,amsfonts,amssymb,amsthm}
\usepackage{latexsym}
\usepackage{makeidx}

\usepackage{vmargin}
\setmarginsrb{18mm}{16mm}{18mm}{25mm}%
             {10mm}{7mm}{10mm}{10mm}

\numberwithin{equation}{section}

\numberwithin{table}{section}

\usepackage{caption}

\usepackage[all]{xy}

\newtheorem{em-deff}{Definition}[section]
\newtheorem{lemma}[em-deff]{Lemma}
\newtheorem{theorem}[em-deff]{Theorem}
\newtheorem{corollary}[em-deff]{Corollary}

\newtheorem{proposition}[em-deff]{Proposition}
\newtheorem{em-fact}[em-deff]{Fact}
\newtheorem{em-example}[em-deff]{Example}

\newtheorem{problem}[em-deff]{Problem}

\newtheorem{em-remark}[em-deff]{Remark}
\newtheorem{question}[em-deff]{Question}

\newenvironment{example}{\begin{em-example} \em }{ \end{em-example}}
\newenvironment{remark}{\begin{em-remark} \em }{ \end{em-remark}}
\newenvironment{deff}{\begin{em-deff} \em }{ \end{em-deff}}
\newenvironment{fact}{\begin{em-fact} \em }{ \end{em-fact}}

\newcommand{\N}{\mathbb N}
\newcommand{\Z}{\mathbb Z}
\newcommand{\Q}{\mathbb Q}

\newcommand{\J}{\mathbb J}

\def\f{\varphi}
\def\ent{\mathrm{ent}}
\def\supp{\mathrm{supp}}
\def\Per{\mathrm{Per}}
\def\End{\mathrm{End}}

\def\cont{\mathfrak c}
\def\sc{\mathrm{sc}}

\title{Strings of group endomorphisms}

\author{
Dikran Dikranjan 
\\{\footnotesize {\tt  dikran.dikranjan@dimi.uniud.it}} 
\\{\footnotesize Dipartimento di Matematica e Informatica,}
\\{\footnotesize Universit\`{a} di Udine,}
\\{\footnotesize Via delle Scienze, 206 - 33100 Udine, Italy} 
\and
Anna Giordano Bruno 
\\{\footnotesize {\tt  anna.giordanobruno@math.unipd.it}} 
\\{\footnotesize Dipartimento di Matematica Pura e Applicata,}
\\{\footnotesize Universit\`a di Padova,}
\\{\footnotesize Via Trieste, 63 - 35121 Padova}
\and 
Simone Virili
\\{\footnotesize {\tt simone.virili@gmail.it}}
\\{\footnotesize Dipartimento di Matematica Pura e Applicata,}
\\{\footnotesize Universit\`a di Padova,}
\\{\footnotesize Via Trieste, 63 - 35121 Padova}
}

\date{}

\begin{document}

\maketitle


\abstract{Recently the strings and the string number of self-maps were used in the computation of the algebraic entropy of special group endomorphisms. We introduce two special kinds of strings, and their relative string numbers. We show that a dichotomy holds for all these three string numbers; in fact, they admit only zero and infinity as values on group endomorphisms.}

\section{Introduction}

The left Bernoulli shift and the two-sided Bernoulli shift are relevant examples for both ergodic theory and topological dynamics, while the right Bernoulli shift is fundamental for the theory of algebraic entropy (see Section \ref{ent-sec}, and for more details and properties see \cite{DGSZ}). We start giving their definitions. Let $K$ be an abelian group.
\begin{itemize}
\item[(a)]  The \emph{two-sided Bernoulli shift} ${}^t\!{\beta}_K$ of the group $K^{\mathbb Z}$ is  defined by 
$$ {}^t\!\beta_K((x_n)_{n\in\mathbb Z})=(x_{n-1})_{n\in\mathbb Z}, \mbox{ for } (x_n)_{n\in\mathbb Z}\in K^{\mathbb Z}.$$
\item[(b)] The \emph{right  Bernoulli shift} $\beta_K$ and the \emph{left Bernoulli shift} $_K\beta$ of the group $K^{\mathbb N}$ are defined respectively by 
$$\beta_K(x_1,x_2,x_3,\ldots)=(0,x_1,x_2,\ldots) \ \mbox{and}\ {}_K\beta(x_0,x_1,x_2,\ldots)=(x_1,x_2,x_3,\ldots).$$ 
\end{itemize}
Moreover, let ${}^t\!\beta_K^\oplus={}^t\!\beta_K\restriction_{K^{(\Z)}}$, $\beta_K^\oplus=\beta_k\restriction_{K^{(\N)}}$ and ${}_K\beta^\oplus={}_K\beta\restriction_{K^{(\N)}}$ be the respective restrictions on the direct sums.

The next definition given in \cite{AKH} is inspired by these classical notions of shift. 
For every set $\Gamma$, every self-map $\lambda:\Gamma\to\Gamma$ and for every abelian group $K$, the \emph{generalized shift} $\sigma_\lambda$ of the direct product $K^{\Gamma}$ is the endomorphism 
\begin{center}
$\sigma_\lambda:K^{\Gamma}\to K^{\Gamma}$, defined by $(x_i)_{i\in\Gamma}\mapsto(x_{\lambda(i)})_{i\in\Gamma}$.
\end{center}
The name of the generalized shifts comes from the fact that they generalize in an obvious sense the Bernoulli shifts; indeed, the left and the two-sided Bernoulli shifts are generalized shifts, while the right Bernoulli shift can be well approximated by a generalized shift of the same algebraic entropy (see \cite{AADGH}). In \cite{AADGH} the restriction $\sigma_\lambda^\oplus:=\sigma_\lambda\restriction_{K^{(\Gamma)}}:K^{(\Gamma)}\to K^{(\Gamma)}$ of $\sigma_\lambda$ to the direct sum $K^{(\Gamma)}$ was considered. Since $K^{(\Gamma)}$ is a $\sigma_\lambda$-invariant subgroup of $K^\Gamma$ if and only if $\lambda$ is a finite-to-one self-map of $\Gamma$ (that is, $\lambda^{-1}(x)$ is finite for every $x\in\Gamma$), one has to add this hypothesis in order to study $\sigma_\lambda^\oplus$ as an endomorphism of $K^{(\Gamma)}$.

\medskip
Other basic concepts related to dynamical arguments are the following. For a set $\Gamma$, a self-map $\lambda:\Gamma\to\Gamma$ and $x\in\Gamma$ the orbit of $x$ for $\lambda$ is $O(x)=\{\lambda^n(x):n\in\N\}$.

\begin{deff}
Let $\Gamma$ be a set and $\lambda:\Gamma\to \Gamma$ a self-map. 
\begin{itemize}
\item[(a)]An element $x\in \Gamma$ is a \emph{periodic} point of $\lambda$ if there exists an integer $n\in\N_+$ such that $\lambda^{n}(x)=x$. We will denote by $\Per(\lambda)$ the set of all periodic points of $\lambda$. Moreover, $\lambda$ is said to be \emph{locally periodic} if every $x\in\Gamma$ is a periodic point of $\lambda$ (i.e., $\Per(\lambda)=\Gamma$). Finally, $\lambda$ is said to be \emph{periodic} if there exists $n\in\N_+$ such that $\lambda^n(x)=x$ for every $x\in\Gamma$.
\item[(b)]An element $x\in \Gamma$ is a \emph{quasi-periodic} point of $\lambda$ if there exist $n< m$ in $\N$ such that $\lambda^{n}(x)=\lambda^{m}(x)$. We will denote by $Q\Per(\lambda)$ the set of all quasi-periodic points of $\lambda$. Moreover, $\lambda$ is said to be \emph{locally quasi-periodic} if every $x\in\Gamma$ is a quasi-periodic point of $\lambda$ (i.e., $\Gamma=Q\Per(\lambda)$). Finally, $\lambda$ is said to be \emph{quasi-periodic} if there exist $n<m$ in $\N$ such that $\lambda^n=\lambda^m$.
\end{itemize}
\end{deff}

\subsection{The string numbers}

The notion of string given in item (b) of the following definition was introduced in \cite{AADGH} in order to compute the algebraic entropy of the generalized shifts. We give first in item (a) a weaker concept, that will be used in the paper, and in item (c) we define a special class of strings which, after iterations of the map, ``form a loop''. 

\begin{deff}
Let $\Gamma$ be a set and $\lambda:\Gamma\to\Gamma$ a self-map. A sequence $S=\{x_n\}_{n\in\N}\subseteq \Gamma$ is
\begin{itemize}
\item[(a)] a \emph{pseudostring} of $\lambda$ if $\lambda(x_n)=x_{n-1}$ for every $n\in\N_+$ ($x_0$ is called first term);
\item[(b)] a \emph{string} of $\lambda$ if $S$ is a pseudostring such that the $x_n$'s are pairwise distinct elements;
\item[(c)] a \emph{singular} string  of $\lambda$ if $S$ is a string of $\lambda$ such that $x_0\in Q\Per(\lambda)$. 
\end{itemize}
\end{deff}

Note that a string $S$ such that there exists $k\in\N_+$ with $\lambda^k(x_0)\in S$ is a singular string. Clearly, a string $S$ of $\lambda$ in $\Gamma$ is singular if and only if $S\subseteq Q\Per(\lambda)$.

\smallskip
Considering an endomorphism of an abelian group, instead of a self-map of a set, it is possibile to introduce also the following special kind of singular strings.

\begin{deff}
If $\Gamma$ is an abelian group and $\lambda:\Gamma\to\Gamma$ is an endomorphism, a string $S=\{x_n\}_{n\in\N}$ is
\begin{itemize}
\item[(d)] a \emph{null string}, if $x_0\neq0$ and $\lambda^k(x_0)=0$ for some $k\in\N_+$.
\end{itemize}
\end{deff}

In \cite{AADGH} a cardinal function was defined to measure the number of pairwise disjoint strings of a self-map; indeed, a precise formula for the algebraic entropy of a generalized shift $\sigma_\lambda^\oplus$ was found, making use of the string number of $\lambda$ and its properties (see Section \ref{ent-sec} for the precise formula).
We recall the definition of string number in (b$'$), while in (c$'$) and (d$'$) we introduce similar cardinal functions that measure the number of pairwise disjoint non-singular strings of a self-map and the number of pairwise disjoint null strings of a group endomorphism respectively. In general we call these three cardinal functions ``string numbers''. 

\begin{deff}
Let $\Gamma$ be a set and $\lambda:\Gamma\to\Gamma$ a self-map. Let
\begin{itemize}
\item[(b$'$)] $s(\lambda)=\sup\left\{|\mathcal F|: \mathcal F\ \text{is a family of pairwise disjoint strings of $\lambda$ in $\Gamma$}\right\}$, the \emph{string number} of $\lambda$;
\item[(c$'$)] $ns(\lambda)=\text{sup}\left\{|\mathcal{F}|:\text{ $\mathcal{F}$ is a family of pairwise disjoint non-singular strings of $\lambda$ in $\Gamma$}\right\}$, the \emph{non-singular string number} of $\lambda$;
\item[(d$'$)] $s_0(\lambda)=\sup\left\{|\mathcal F|: \mathcal F\ \text{is a family of pairwise disjoint null strings of $\lambda$ in $\Gamma$}\right\}$, the \emph{null string number} of $\lambda$, if $\Gamma$ is an abelian group and $\lambda$ an endomorphism of $\Gamma$.
\end{itemize}
\end{deff}

As for the algebraic entropy, we want these string numbers to have values in $\N\cup\{\infty\}$; so, when the suprema in these definitions are infinite, we impose that they equal $\infty$, with the usual convention that $a+\infty=\infty$ for every $a\in\N\cup\{\infty\}$.

\medskip
The aim of this paper is to study the properties of these three cardinal functions, the relations among them, and to measure their values for endomorphisms of abelian groups. 

\medskip
It is easy to observe that $s(\lambda)\geq\max\{ns(\lambda),s_0(\lambda)\}$, for a group endomorphism $\lambda$. More precisely, Theorem \ref{s=ns+s0} shows the following relation among the three string numbers:
\begin{equation}\label{s=ns+s0-eq}
s(\lambda)=ns(\lambda)+s_0(\lambda).
\end{equation}

The following Theorem A characterizes the group endomorphisms admitting some string.
It gives an affirmative answer to \cite[Problem 6.6(a)]{AADGH}, which asked whether $s(\lambda)>0$ implies $s(\lambda)=\infty$ for any endomorphism $\lambda$ of abelian groups.

\smallskip
It will be shown in Section \ref{sur-core} that surjectivity is not a relevant restriction in the study of the string numbers. Indeed, for a set $\Gamma$ and a self-map $\lambda:\Gamma\to\Gamma$ we construct the maximum subset $\sc(\lambda)$ of $\Gamma$ on which the restriction of $\lambda$ is surjective, namely, the \emph{surjective core} of $\lambda$; moreover, all the strings of $\lambda$ are contained in the surjective core of $\lambda$ (see Proposition \ref{s-c}).

\medskip
\noindent{\bf Theorem A.} {\em 
Let $G$ be an abelian group and $\f\in\End(G)$. The following conditions are equivalent:
\begin{itemize}
\item[(i)] $s(\f)>0$;
\item[(ii)] $s(\f)=\infty$;
\item[(iii)] $\sc(\f)\not\subseteq\Per(\f)$. 
\end{itemize}}
\medskip

The following corollary characterizes the \emph{surjective} group endomorphisms of finite string number, showing that these are exactly the locally periodic ones.

\medskip
\noindent{\bf Corollary.}
{\em Let $G$ be an abelian group and $\f\in\End(G)$ surjective. The following conditions are equivalent:
\begin{itemize}
\item[(i)] $s(\f)=0$;
\item[(ii)] $s(\f)<\infty$;
\item[(iii)] $G=\Per(\f)$.\hfill $\qed$
\end{itemize}}
\medskip

Since a locally periodic endomorphism of a finitely generated abelian group is periodic, we obtain the following nice characterization of periodic endomorphisms of finitely generated abelian groups, which answers a question of Zanolin and Corvaja.
\begin{quote}
{\em Let $G$ be a finitely generated abelian group and $\f\in\End(G)$ surjective. Then $s(\f)=0$ if and only if $\f$ is periodic.}
\end{quote}
In other words, if $\f$ is not periodic, then $\f$ admits infinitely many strings.

\medskip
Theorems A$^*$ and A$^{**}$ in the sequel are the counterpart of Theorem A for the non-singular string number and the null string number respectively.

\medskip
\noindent{\bf Theorem A$^*$.} 
{\em Let $G$ be an abelian group and $\f\in\End(G)$. The following conditions are equivalent:
\begin{itemize}
\item[(i)] $ns(\f)>0$;
\item[(ii)] $ns(\f)=\infty$;
\item[(iii)] $\sc(\f)\not\subseteq Q\Per(\f)$. 
\end{itemize}}
\medskip

Analogously to the Corollary of Theorem A, we give the following consequence of Theorem A$^*$ for \emph{surjective} group endomorphisms $\f$, showing that $ns(\f)=0$ if and only if $\f$ is locally quasi-periodic.

\medskip
\noindent{\bf Corollary.}
{\em Let $G$ be an abelian group and $\f\in\End(G)$ surjective. The following conditions are equivalent:
\begin{itemize}
\item[(i)] $ns(\f)=0$;
\item[(ii)] $ns(\f)<\infty$;
\item[(iii)] $G= Q\Per(\f)$.\hfill $\qed$
\end{itemize}}
\medskip

To a certain extent, the next theorem goes in the opposite direction with respect to Theorem A$^*$, since the null strings are necessarily singular. For an abelian group $G$ and $\f\in\End(G)$, let $\ker_\infty\f=\bigcup_{n\in\N}\ker(\f^n)$ be the \emph{hyperkernel} of $\f$. The null strings $S$ of $\f$ are precisely the strings of $\f$ contained in $\ker_\infty\f$ (see Lemma \ref{null<->iper}).

\medskip
\noindent{\bf Theorem A$^{**}$.} 
{\em Let $G$ be an abelian group and $\f\in\End(G)$. The following conditions are equivalent:
\begin{itemize}
\item[(i)] $s_0(\f)>0$;
\item[(ii)] $s_0(\f)=\infty$;
\item[(iii)] $\sc(\f)\cap\ker_\infty\f\neq0$.
\end{itemize}}
\medskip
 
If $\f$ is surjective, then $\sc(\f)=G$, hence $\sc(\f)\cap \ker_\infty\f=G\cap \ker_\infty\f=\ker_\infty\f$. This gives the following corollary of Theorem A$^{**}$.

\medskip
\noindent{\bf Corollary.}
{\em Let $G$ be an abelian group and $\f\in\End(G)$ surjective. The following conditions are equivalent:
\begin{itemize}
\item[(i)] $s_0(\f)=0$;
\item[(ii)] $s_0(\f)<\infty$;
\item[(iii)] $\f$ is injective. \hfill $\qed$
\end{itemize}}
\medskip

\subsection{Algebraic entropy and string numbers}\label{ent-sec}

The strings and the string number arose in \cite{AADGH} in the computation of the algebraic entropy of generalized shifts. Here we recall the definition of algebraic entropy, given by Adler, Konheim and McAndrew \cite{AKM} and Weiss \cite{W}, and compare the basic properties of the algebraic entropy with the properties of the string numbers.

\medskip
Let $G$ be an abelian group and $F$ a finite subgroup of $G$; for an endomorphism $\f:G\to G$ and a positive integer $n$, let $T_n(\f,F)=F+\f(F)+\ldots+\f^{n-1}(F)$ be the \emph{$n$-th $\f$-trajectory} of $F$ with respect to $\f$. The \emph{algebraic entropy of $\f$ with respect to $F$} is $$H(\f,F)={\lim_{n\to \infty}\frac{\log|T_n(\f,F)|}{n}},$$ and the \emph{algebraic entropy} of $\f$ is $$\ent(\f)=\sup\{H(\f,F): F\ \text{is a finite subgroup of } G\}.$$

\smallskip
First we recall the precise formula found in \cite{AADGH} for the algebraic entropy of a generalized shift $\sigma_\lambda^\oplus:K^{(\Gamma)}\to K^{(\Gamma)}$, with $\lambda:\Gamma\to \Gamma$ a finite-to-one self-map of $\Gamma$, that is, $$\ent(\sigma_\lambda^\oplus)=s(\lambda)\cdot \log|K|.$$ Moreover, in \cite{GB} the algebraic entropy of a generalized shift $\sigma_\lambda$ of the direct product $K^\Gamma$ was computed, and also in this case the string number played a central role.

\medskip
The following are the basic properties of the algebraic entropy.

\begin{fact}\label{ent}\cite{DGSZ,W}
Let $G$ be an abelian group and $\f\in\End(G)$.
\begin{itemize}
\item[(A)](Conjugation by isomorphism) If $\f$ is conjugated to an endomorphism $\psi:H\to H$, of another abelian group $H$, by an isomorphism, then $\ent(\f) = \ent(\psi)$.
\item[(B)](Logarithmic law) For every non-negative integer $k$, $\ent(\f^k) = k \cdot \ent(\f)$. If $\f$ is an automorphism, then $\ent(\f^k)=|k|\cdot\ent(\f)$ for every integer $k$.
\item[(C)](Addition Theorem) If $G$ is torsion and $H$ is a $\f$-invariant subgroup of $G$, then $\ent(\f)=\ent(\f\restriction_H)+\ent(\overline\f)$, where $\overline\f:G/H\to G/H$  is the endomorphism induced by $\f$.
\item[(D)](Continuity for direct limits) If $G$ is direct limit of $\f$-invariant subgroups $\{G_i : i \in I\}$, then $\ent(\f)=\sup_{i\in I}\ent(\f\restriction_{G_i})$.
\item[(E)](Uniqueness) The algebraic entropy of the endomorphisms of the torsion abelian groups is characterized as the unique collection $h = \{h_G : G \text{ torsion abelian group}\}$ of functions $h_G:\End(G) \to \mathbb R_+$ that satisfy (A), (B), (C), (D) and $h_{\Z(p)^{(\N)}}(\beta^\oplus_{\Z(p)})=\log|\Z(p)|$ for every prime $p$.
\end{itemize}
\end{fact}

Some of these basic properties of the algebraic entropy have counterparts for the string numbers; indeed, the string numbers are stable under taking conjugation by group isomorphisms (see Lemma \ref{cbi}), and also a logarithmic law holds for the string numbers (see Corollary \ref{s>s*}).

Note that the validity of the Addition Theorem for a cardinal function implies monotonicity of the cardinal function for endomorphisms $\f$ of abelian groups $G$ under taking restrictions $\f\restriction_H$ to $\f$-invariant subgroups $H$ and under taking induced maps $\overline\f$ on quotients $G/H$. The string numbers are monotone under taking restrictions to invariant subgroups (see Lemma \ref{subgroups}). Moreover, the string number and the non-singular string number are monotone under taking the induced map $\overline\f$ on the quotient $G/H$ if $\f$ is surjective (see Theorem \ref{quotients}), while the null string number fails to have this property (see Example \ref{s0-non-monotone}). In general, all the three string numbers do not obey this monotonicity law (see Example \ref{Jp} for the string number and the non-singular string number), and so the Addition Theorem holds for none of the string numbers (see also Example \ref{CorZan}).

\smallskip
Furthermore, it is proved in \cite{DGSZ} that for torsion abelian groups the condition of local quasi-periodicity in item (iii) of the Corollary of Theorem A$^*$ is equivalent to $\ent(\f)=0$, and so we have the following new characterization of surjective group endomorphisms of algebraic entropy zero in terms of the non-singular string number. 

\medskip
\noindent{\bf Corollary.}
\emph{Let $G$ be a torsion abelian group and $\f\in\End(G)$ surjective. The following conditions are equivalent:
\begin{itemize}
\item[(i)] $ns(\f)=0$;
\item[(ii)] $ns(\f)<\infty$;
\item[(iii)] $\ent(\f)=0$.\hfill $\qed$
\end{itemize}}

Note that this corollary does not imply $ns(\f)=\ent(\f)$ for surjective endomorphisms of torsion abelian groups --- see the last row of Table \ref{table} below.
Nevertheless, if $G$ is a torsion abelian group and $\f\in\End(G)$ is surjective, then $$s(\f)\geq ns(\f)\geq \ent(\f).$$

\smallskip
Finally, the string numbers of the Bernoulli shifts are quite different from the algebraic entropy of the Bernoulli shifts; we calculate the values of the string numbers of the Bernoulli shifts in Example \ref{Example:shift}, while for the values of the algebraic entropy we refer to \cite{DGSZ}. We collect this information in Table \ref{table}, where we add also the values of the adjoint algebraic entropy on the Bernoulli shifts. 
The adjoint algebraic entropy was defined in \cite{DGS} substituting in the definition of the algebraic entropy the family of all finite subgroups with the family of finite-index subgroups. We give the precise definition: if $N$ is a finite-index subgroup of an abelian group $G$, $\f\in\End(G)$, and $n$ is a positive integer, the $n$-th $\f$-cotrajectory of $N$ is
$C_n(\f,N) = \frac{G}{N\cap\f^{-1}N\cap\ldots\cap\f^{-n+1}N}$.
The \emph{adjoint algebraic entropy of $\f$ with respect to $N$} is
$$H^\star(\f,N)=\lim_{n\to\infty}\frac{\log|C_n(\f,N)|}{n},$$
and the \emph{adjoint algebraic entropy} of $\f$ is
$$\ent^\star(\f)=\sup\{H^\star(\f,N): N\leq G,\ G/N\ \text{finite}\}.$$

\smallskip
Theorems A, A$^*$ and A$^{**}$ show that there is a dichotomy for the values of the three string numbers, which can be either zero or infinity. Since the same dichotomy holds for the values of the adjoint algebraic entropy, it is worthwhile to compare each of the string numbers also with the adjoint algebraic entropy.

\begin{center}
\begin{tabular}{|c|ccc|cc|}
\hline
 & $s(-)$ & $ns(-)$ & $s_0(-)$ & $\ent(-)$ & $\ent^\star(-)$ \\
 \hline
$\beta^\oplus_K$ & $0$  & $0$ & $0$ & $\log|K|$ & $\infty$ \\
${}_K\beta^\oplus$ & $\infty$ & $0$ & $\infty$ & $0$ & $\infty$ \\
$\overline\beta_K^\oplus$ & $\infty$ & $\infty$ & $0$ & $\log|K|$ & $\infty$ \\
\hline
\end{tabular}
\captionof{table}{values on Bernoulli shifts}\label{table}
\end{center}

\bigskip
In analogy with what is done for the algebraic entropy in \cite{DGSZ} and for the adjoint algebraic entropy in \cite{DGS}, we introduce the following notions of string numbers of an abelian group, noting that in this case it is sufficient to distinguish between value zero and value infinity, in view of Theorems A, A$^*$ and A$^{**}$.

\begin{deff}
Let $G$ be an abelian group. 
\begin{itemize}
\item[(b$''$)]The \emph{string number} of $G$ is $s(G)=\sup\{s(\f):\f\in\End(G)\}$.
\item[(c$''$)]The \emph{non-singular string number} of $G$ is $ns(G)=\sup\{ns(\f):\f\in\End(G)\}$.
\item[(d$''$)]The \emph{null string number} of $G$ is $s_0(G)=\sup\{s_0(\f):\f\in\End(G)\}$.
\end{itemize}
\end{deff}

We leave open the following problem, which will be discussed in \cite{DGV2}.

\begin{problem}
Describe the abelian groups $G$ that have $s(G)=0$ (respectively, $ns(G)=0$, $s_0(G)=0$).
\end{problem}

\subsubsection*{Acknowledgements}
It is a pleasure to thank the referee for her/his constructive criticism.

\section{General properties}

\subsection{The surjective core of a self-map}\label{sur-core}

For every self-map $\lambda: \Gamma \to \Gamma$ of a set $\Gamma$ one can be interested in those restrictions of $\lambda$ that are \emph{surjective}. In other words, one can consider $\lambda$-invariant subsets $\Lambda$ of $\Gamma$ such that the restriction $\lambda\restriction_\Lambda: \Lambda \to \Lambda$ is surjective. Such non-empty subsets $\Lambda$ need not exist, as the following example shows. This example motivates also the second part of Lemma \ref{sc-existence}.

\begin{example}
Let $\lambda:\N\to\N$ be defined by $\lambda(n)=n+1$ for every $n\in\N$. Then $\lambda\restriction_\Lambda:\Lambda\to\Lambda$ is not surjective for every non-empty $\lambda$-invariant subset $\Lambda$ of $\N$.
\end{example}

Nevertheless, something is easy to prove right away: there exists a biggest such $\Lambda$ (that can be empty), that we denote by $\sc(\lambda)$ and call \emph{surjective core of $\lambda$}. We denote by $\lambda^\sc$ the restriction of $\lambda$ to $\sc(\lambda)$. 
In the following lemma we give a constructive proof of the existence of the surjective core of a self-map.

\begin{lemma}\label{sc-existence}
Let $\Gamma$ be a set and $\lambda:\Gamma\to \Gamma$ a self-map. The surjective core $\sc(\lambda)$ of $\lambda$ exists.

If $\Gamma$ is a group and $\lambda\in\End(\Gamma)$, then $\sc(\lambda)$ is a subgroup of $\Gamma$, hence it is not empty.
\end{lemma}
\begin{proof}
We define by transfinite induction a transfinite decreasing chain. Let $\lambda^0(\Gamma)=\Gamma$. If $\alpha=\beta+1$ is a successor ordinal, then $\lambda^\alpha(\Gamma)=\lambda(\lambda^\beta(\Gamma))$, if $\alpha$ is a limit ordinal, then $\lambda^\alpha(\Gamma)=\bigcap_{\beta<\alpha}\lambda^\beta(\Gamma)$. Since this is a decreasing chain of subsets, it stabilizes, that is, there exists $\alpha_0$ such that $\lambda^\beta(\Gamma)=\lambda^{\alpha_0}(\Gamma)$ for every $\beta\geq\alpha_0$. We define $\sc(\lambda)=\lambda^{\alpha_0}(\Gamma)$; it is easy to check that this works. Obviously, if $\Gamma$ is a group and $\lambda\in\End(\Gamma)$, then $\sc(\lambda)$ is a subgroup of $\Gamma$, so it cannot be empty.
\end{proof}

The next result permits to reduce the study of the string numbers to the case of surjective endomorphisms.

\begin{proposition}\label{s-c}
Let $\Gamma$ be a set and $\lambda:\Gamma\to\Gamma$ a self-map. Then $s(\lambda) = s(\lambda^\sc)$ and $ns(\lambda)=ns(\lambda^\sc)$; moreover, if $\Gamma$ is an abelian group and $\lambda\in\End(\Gamma)$, then $s_0(\lambda)=s_0(\lambda^\sc)$. 
\end{proposition}
\begin{proof}
We show by transfinite induction that if $S$ is a string of $\lambda$, then $S$ is contained in $\lambda^\alpha(\Gamma)$ for every cardinal $\alpha$, and this implies that $S$ is contained in $\sc(\lambda)$. If $\alpha=0$, clearly $S\subseteq\Gamma$. Assume that $\alpha=\beta+1$ is a successor ordinal and that $S\subseteq \lambda^\beta(\Gamma)$; since $S\subseteq\lambda(S)$, it follows that $S\subseteq \lambda(\lambda^\beta(\Gamma))=\lambda^\alpha(\Gamma)$. If $\alpha$ is a limit cardinal and $S\subseteq \lambda^\beta(\Gamma)$ for every $\beta<\alpha$, then $S\subseteq \lambda^\alpha(\Gamma)$. Hence $S\subseteq \sc(\lambda)$.
\end{proof}

We give now an example of calculation of the surjective core of a specific group endomorphism.

\begin{example}\label{sc-ex}
Let $p$ be a prime, let $G$ be an abelian group, and let $\mu_p\in\End(G)$ be the multiplication by $p$, that is, the endomorphism of $G$ defined by $\mu_p (x) = px$ for every $x\in G$. 
\begin{itemize}
\item[(a)] The surjective core of $\mu_p$ is $d_p(G)$, the maximum $p$-divisible subgroup of $G$. By Proposition \ref{s-c}, if $d_p(G)=0$ (i.e., $G$ is $p$-reduced), then $s(\mu_p)=0$.
\item[(b)] By (a), if $G$ is an abelian $p$-group, then the surjective core of $\mu_p$ is exactly $d(G)$, the maximum divisible subgroup of $G$; so, if $G$ is reduced, then $s(\mu_p)=0$.
\item[(c)] By (a), if $G$ is a torsion-free abelian group, then the surjective core of $\mu_p$ is precisely $p^\omega G$; so, if $p^\omega G = 0$, then $s(\mu_p) =0$.
\end{itemize}
\end{example}

The left Bernoulli shift is surjective and the two-sided Bernoulli shift is an automorphism, so their surjective cores coincide with their domains. In the following example we consider the right Bernoulli shift, which is not surjective, and which turns out to have trivial surjective core.

\begin{example}\label{sc-bernoulli}
Let $K=\Z(p)$ for some prime $p$, and consider the right Bernoulli shift $\beta_K^\oplus:K^{(\N)}\to K^{(\N)}$.
\begin{itemize}
\item[(a)] Then $\sc(\beta_K^\oplus)=0$.
\item[(b)] Consider the subgroup $H=\{x=(x_n)_{n\in\N}:\sum_{n\in\N}x_n=0\}$ of $K^{(\N)}$. Then the endomorphism $\overline\beta_K^\oplus:K^{(\N)}/H\to K^{(\N)}/H$ induced by $\beta_K^\oplus$ is the identity. Therefore $\sc(\overline\beta_K^\oplus)=K^{(\N)}/H$.
\end{itemize}
\end{example}

Item (b) of this example shows that the projection on the quotient of the surjective core of a group endomorphism can be strictly contained in the surjective core of the endomorphism induced on the quotient.

\subsection{Basic properties of strings and of string numbers}

The following lemma is a useful criterion in order to verify whether two (pseudo)strings are disjoint.

\begin{lemma}\label{disgiunte}
Let $G$ be an abelian group, $\f\in\End(G)$, $S=\{x_n\}_{n\in\N}$ and $S'=\{y_n\}_{n\in\N}$ two pseudostrings of $\f$ in $G$. Then $S$ and $S'$ are disjoint if and only if $x_0\notin S'$ and $y_0\notin S$.
\end{lemma}
\begin{proof}
If the pseudostrings $S$ and $S'$ are disjoint, it follows immediately that $x_0\notin S'$ and $y_0\notin S$. Viceversa, suppose that $S$ and $S'$ meet non trivially. So $x_n=y_k$ for some $n,k\in\N$.
If $n\leq k$, then $x_0=\f^n(x_n)=\f^n(y_k)=y_{k-n}\in S'$. 
If $k\leq n$, then $y_0=\f^k(y_k)=\f^k(x_n)=x_{n-k}\in S$.
\end{proof}

Let $G$ be an abelian group, $\f\in\End(G)$, $H$ a $\f$-invariant subgroup of $G$, $S=\{x_n\}_{n\in\N}$ a pseudostring of $\f$ and $\pi:G\to G/H$ the canonical projection. Then $\pi(S)=\{\pi(x_n)\}_{n\in\N}$ is a pseudostring of $\overline\f$, where $\overline \f:G/H\to G/H$ is the endomorphism induced by $\f$. 

\begin{lemma}\label{quoziente}
Let $G$ be an abelian group, $\f\in\End(G)$, $H$ a $\f$-invariant subgroup of $G$, $\pi:G\to G/H$ the canonical projection and $\overline \f:G/H\to G/H$ the endomorphism induced by $\f$. Let $S=\{x_n\}_{n\in\N}$ and $S'=\{y_n\}_{n\in\N}$ be two pseudostrings of $\f$. 
\begin{itemize}
\item[(a)] If $\pi(S)=\{\pi(x_n)\}_{n\in\N}$ is a string of $\overline \f$, then $S$ is a string of $\f$.
\item[(b)] If $\pi(S)\cap\pi(S')=\emptyset$, then $S\cap S'=\emptyset$. \hfill$\qed$
\end{itemize}
\end{lemma}



\begin{lemma}\label{cbi}
Let $\Gamma$ and $\Lambda$ be two sets, and $\xi:\Gamma\to \Lambda$ a bijection. Let $\lambda:\Gamma\to \Gamma$ and $\eta:\Lambda\to\Lambda$ be self-maps such that $\lambda=\xi^{-1} \eta\xi$. Then $s(\lambda)=s(\eta)$ and $ns(\lambda)=ns(\eta)$; if $\Gamma$ and $\Lambda$ are abelian groups, $\xi$ is an isomorphism, and $\lambda$ and $\eta$ are endomorphisms, then $s_0(\lambda)=s_0(\eta)$.  \hfill$\qed$
\end{lemma} 


\begin{lemma}\label{subgroups}
Let $\Gamma$ be a set, $\lambda:\Gamma\to \Gamma$ a self-map and $\Lambda$ a $\lambda$-invariant subset of $\Gamma$. Then $s(\lambda)\geq s(\lambda\restriction_\Lambda)$ and $ns(\lambda)\geq ns(\lambda\restriction_\Lambda)$; if $\Gamma$ is an abelian group, $\lambda\in\End(\Gamma)$ and $\Lambda$ a subgroup of $\Gamma$, also $s_0(\lambda)\geq s_0(\lambda\restriction_\Lambda)$. \hfill$\qed$
\end{lemma}

The following lemma gives an easy criterion in order to verify whether a pseudostring contains a string, namely, it suffices to see that it is infinite.

\begin{lemma}\label{infinite->string}
Let $\Gamma$ be a set, $\lambda:\Gamma\to \Gamma$ a self-map and $S=\{x_n\}_{n\in\N}$ a pseudostring of $\lambda$. If $S$ is infinite, then $S$ contains a string, and is itself a string when $\lambda$ is injective.
\end{lemma}
\begin{proof}
If $S$ itself is not a string, there exist $n<m$ in $\N$ such that $x_n=x_m$. Therefore $x_m$ is a periodic point of $\lambda$ of order $m-n$. Since $S$ is infinite, there exists a maximal $m\in\N$ such that $x_m$ is a periodic point of $\lambda$. Then $x_{m+1}\not\in O(x_m)$, and so $S'=\{x_n\}_{n>m}$ is a string contained in $S$. 
\end{proof}

\begin{lemma}\label{s=0<->s=0}
Let $\Lambda_1,\Lambda_2$ be sets, $\lambda_i:\Lambda_i\to\Lambda_i$ self-maps for $i=1,2$, $\Gamma=\Lambda_1\times\Lambda_2$ and $\lambda=\lambda_1\times\lambda_2$. Then:
\begin{itemize}
\item[(a)] $s(\lambda)=0$ if and only if $s(\lambda_1)=s(\lambda_2)=0$;
\item[(b)] $ns(\lambda)=0$ if and only if $ns(\lambda_1)=ns(\lambda_2)=0$.
\end{itemize} 
If $\Gamma$ is an abelian group, $\Lambda_1,\Lambda_2$ subgroups of $\Gamma$ and $\lambda\in\End(\Gamma)$, then
\begin{itemize}
\item[(c)] $s_0(\lambda)=0$ if and only if $s_0(\lambda_1)=s_0(\lambda_2)=0$.
\end{itemize}
\end{lemma}
\begin{proof}
(a) If $s(\lambda)=0$, then $s(\lambda_1)=s(\lambda_2)=0$ by Lemma \ref{subgroups}.

Assume that $s(\lambda)>0$. Let $S=\{x_n\}_{n\in\N}$ be a string of $\lambda$. For every $n\in\N$ let $x_n=(x_n^1,x_n^2)$, where $x_n^1\in\Lambda_1$ and $x_n^2\in\Lambda_2$. Then $S_1:=\{x_n^1\}_{n\in\N}$ and $S_2:=\{x_n^2\}_{n\in\N}$ are pseudostrings of $\lambda_1$ and $\lambda_2$ respectively. Assume that $S_1$ and $S_2$ contains no string. Then $S_1$ and $S_2$ are finite by Lemma \ref{infinite->string}. It follows that $S\subseteq S_1\times S_2$ is finite, and hence $S$ cannot be a string, which is a contradiction. Consequently either $S_1$ or $S_2$ contains a string, that is either $s(\lambda_1)>0$ or $s(\lambda_2)>0$.

\smallskip
(b) If $ns(\lambda)=0$, then $ns(\lambda_1)=ns(\lambda_2)=0$ by Lemma \ref{subgroups}.

Assume that $ns(\lambda)>0$. Let $S=\{x_n\}_{n\in\N}$ be a non-singular string of $\lambda$. For every $n\in\N$ let $x_n=(x_n^1,x_n^2)$, where $x_n^1\in\Lambda_1$ and $x_n^2\in\Lambda_2$. Then $S_1:=\{x_n^1\}_{n\in\N}$ and $S_2:=\{x_n^2\}_{n\in\N}$ are pseudostrings of $\lambda_1$ and $\lambda_2$ respectively. We show that either $S_1$ or $S_2$ is a non-singular string. To this aim, let $x_{-n}^1=\lambda^n(x_0^1)$ and $x_{-n}^2=\lambda^n(x_0^2)$ for every $n\in\N$ and $x_n=(x_n^1,x_n^2)$ for every $n\in\Z$; moreover, assume by contradiction that $x_a^1=x_b^1$ and $x_c^2=x_d^2$ for some non-positive integers $a<b$ and $c<d$. Assume without loss of generality that $a\leq c$ and let $m=(b-a)(d-c)$, which is positive; then $x_a=x_{a-m}$, against the hypothesis that $S$ is a non-singular string. This proves that either $S_1$ or $S_2$ is a non-singular string, that is, either $ns(\lambda_1)>0$ or $ns(\lambda_2)>0$.

\smallskip
(c) If $s_0(\lambda)=0$, then $s_0(\lambda_1)=s_0(\lambda_2)=0$ by Lemma \ref{subgroups}.

Assume that $s_0(\lambda)>0$. Let $S=\{x_n\}_{n\in\N}$ be a null string of $\lambda$. For every $n\in\N$ let $x_n=(x_n^1,x_n^2)$, where $x_n^1\in\Lambda_1$ and $x_n^2\in\Lambda_2$. Then $S_1:=\{x_n^1\}_{n\in\N}$ and $S_2:=\{x_n^2\}_{n\in\N}$ are pseudostrings of $\lambda_1$ and $\lambda_2$ respectively. By the argument in the proof of (a) either $S_1$ or $S_2$ contains a string. Suppose that $S_1$ contains a string. Since $\lambda^m(x_0^1)=0$ for some $m\in\N_+$, $S_1$ is a null string and so $s_0(\lambda_1)>0$. Analogously, if $S_2$ contains a string, then $s_0(\lambda_2)>0$.
\end{proof}



The next result shows that for an injective self-map every string is non-singular, so its string number coincides with its non-singular string number.

\begin{proposition}\label{uguaglianza iniettiva}
Let $\Gamma$ be a set and $\lambda:\Gamma\to\Gamma$ a self-map. If $\lambda$ is injective and $S$ is a string of $\lambda$, then $S$ is non-singular. In particular, $s(\lambda)=ns(\lambda)$.\hfill $\qed$
\end{proposition}

\subsection{Periodicity and string numbers}

In the next result we see that a quasi-periodic self-map has no strings.

\begin{lemma}\label{q-p}
Let $\Gamma$ be a set and $\lambda:\Gamma\to \Gamma$ a self-map. If $\lambda$ is quasi-periodic, then $s(\lambda)=0$. \hfill$\qed$
\end{lemma}

Since a periodic map is quasi-periodic, it follows that a periodic map has no strings. So first examples of group endomorphisms of string number zero are the endomorphism $0_G$ and the identity $id_G$ of an abelian group $G$. Note that in Lemma \ref{q-p} the quasi-periodicity of $\lambda$ cannot be replaced by local quasi-periodicity (see Remark \ref{lqp-strings}).


In item (a) of the following result we show that the string number of a self-map $\lambda$ is zero in case $\lambda$ is locally periodic; then also the non-singular string number is zero. But in item (b) we see that the non-singular string number of $\lambda$ is zero also under the weaker hypothesis that $\lambda$ is locally quasi-periodic.

\begin{proposition}\label{Per>NoStrings}
Let $\Gamma$ be a set and $\lambda:\Gamma\to\Gamma$ a self-map.
\begin{itemize}
\item[(a)] If $\lambda$ is locally periodic, then $s(\lambda)=0$.
\item[(b)] If $\lambda$ is locally quasi-periodic, then $ns(\lambda)=0$.
\end{itemize}
\end{proposition}
\begin{proof}
(a) For every $x\in\Gamma$ let $m_x\in\N_+$ be the minimum natural number such that $\lambda^{m_x}(x)=x$. 
If $\lambda(x)=\lambda(y)$ for some $x,y\in\Gamma$, then $x=\lambda^{m_xm_y}(x)=\lambda^{m_xm_y}(y)=y$ and this proves that $\lambda$ is injective. Moreover, $\lambda(\lambda^{m_x-1}(x))=x$, that shows that $\lambda$ is surjective.
Then $O(x)=\{x,\lambda(x),\ldots,\lambda^{m_x-1}(x)\}$ and $\lambda^{-n}(x)\in O(x)$ for every $n\in\N$.
Since $\lambda$ is injective, each pseudostring of $\lambda$ has to be contained in a single $O(x)$ for some $x\in\Gamma$. We have seen that $O(x)$ is finite, hence all the pseudostrings are finite and they cannot be strings.

\smallskip
(b) is clear. 
\end{proof}

\begin{remark}\label{lqp-strings}
If $\f$ is a locally quasi-periodic group endomorphism, it may happen that $\f$ has strings  (see Example \ref{utilissimo}). These strings are necessarily singular, since $\f$ cannot admit any non-singular string in view of Proposition \ref{Per>NoStrings}(b). 
\end{remark}

The next corollary of Proposition \ref{Per>NoStrings} gives a sufficient condition for $\mu_p$ to have no strings.

\begin{corollary}\label{dptp}
If $G$ is a torsion abelian group and $d_p(t_p(G))=0$ for some prime $p$, then $s(\mu_p)=0$.
\end{corollary}
\begin{proof}
Since $G$ is torsion, $G=\bigoplus_p t_p(G)$. By Proposition \ref{s-c} and Example \ref{sc-ex}(a), $s(\mu_p)=s(\mu_p\restriction_{d_p(G)})$. Since $d_p(t_p(G))=0$ by hypothesis, $$d_p(G)=d_p(t_p(G))\oplus\bigoplus_{q\neq p}t_q(G)=\bigoplus_{q\neq p}t_q(G).$$ Therefore, it suffices to prove that $\mu_p$ is locally periodic in order to apply Proposition \ref{Per>NoStrings}(a) to $\mu_p\restriction_{d_p(G)}$.
Let $x\in d_p(G)$, and consider the finite subgroup $F=\langle x\rangle$ of $G$ of size $m\in\N_+$. Since $F$ is $\mu_p$-invariant and it is finite, $\mu_p\restriction_F:F\to F$ is quasi-periodic. Since $m$ is coprime with $p$, $\mu_p\restriction_F:F\to F$ is an automorphism, and so it is locally periodic. Then $\mu_p:d_p(G)\to d_p(G)$ is locally periodic. Hence apply Propositions \ref{s-c} and \ref{Per>NoStrings}(a) to conclude that $s(\mu_p)=s(\mu_p\restriction_{d_p(G)})=0$.
\end{proof}

Example \ref{sc-ex}(a) shows that $d_p(G)=0$ implies $s(\mu_p)=0$ for an abelian group $G$ and a prime $p$. This implication cannot be reverted; consider for example, for a prime $q$ different from $p$, the group $G=\Z(q)^{(\N)}$ (or $G=\Z(q^\infty)$), which is a torsion abelian group with $d_p(G)=G$, while $d_p(t_p(G))=0$ implies $s(\mu_p)=0$ by Corollary \ref{dptp}.

\medskip
In the following theorem item (a) is the equivalence between (i) and (iii) in Theorem A, and item (b) is the equivalence between (i) and (iii) in Theorem A$^*$.

\begin{theorem}\label{non-per->string}\label{s=0<->G=Per}
Let $G$ be an abelian group and $\f\in\End(G)$ surjective. Then every element $x$ of $G$ is contained as the first term in a pseudostring of $\f$. Furthermore, if $x$ is not periodic, then every pseudostring $S=\{x_n\}_{n\in\N}$ with $x_0=x$ is a string. Consequently:
\begin{itemize}
\item[(a)] $s(\f)=0$ if and only if $\sc(\f)=\Per(\f)$;
\item[(b)] $ns(\f)=0$ if and only if $\sc(\f)\subseteq Q\Per(\f)$.
\end{itemize}
\end{theorem}
\begin{proof}
Since $\f$ is surjective, there exists a pseudostring $S=\{x_n\}_{n\in\N}$ with $x_0=x$ . Assume that $x_n=x_m$ for some $n<m$ in $\N$. Then $x_{m-n}=x$. This means that $x$ is a periodic point of $\f$.

\smallskip
(a) By Proposition \ref{s-c} $s(\f)=s(\f^\sc)$.
Assume that $\sc(\f)\supsetneq \Per(\f)$. Then there exists a non-periodic $x\in\sc(\f)$. By the first part of the theorem there exists a string of $\f^\sc$ with $x$ as first term. In particular $s(\f^\sc)>0$.
Suppose now that $\sc(\f)=\Per(\f)$. Then $\f^\sc$ is locally periodic and so $s(\f^\sc)=0$ by Proposition \ref{Per>NoStrings}(a).

\smallskip
(b) By Proposition \ref{s-c} $ns(\f)=ns(\f^\sc)$.
Assume that $\sc(\f)\not\subseteq Q\Per(\f)$. Then there exists $x\in\sc(\f)$ non-quasi-periodic; in particular $x$ is non-periodic. By the first part of the theorem there exists a string $S$ of $\f^\sc$ with $x$ as first element. Since $x$ is non-quasi-periodic, $S$ is non-singular, and hence $ns(\f^\sc)>0$.
Suppose now that $\sc(\f)\subseteq \Per(\f)$. Then $\f^\sc$ is locally quasi-periodic and so $ns(\f^\sc)=0$ by Proposition \ref{Per>NoStrings}(b).
\end{proof}



Going in the opposite direction with respect to Corollary \ref{dptp}, the following consequence of Theorem \ref{s=0<->G=Per} produces strings of $\mu_p$.

\begin{corollary}\label{dp-notin-t}
Let $G$ be an abelian group and $p$ a prime. If $d_p(G)\not\subseteq t(G)$, then $s(\mu_p)>0$.
\end{corollary}
\begin{proof}
The restriction $\mu_p\restriction_{d_p(G)}:d_p(G)\to d_p(G)$ is surjective by Example \ref{sc-ex}(a). Since $d_p(G)\not\subseteq t(G)$, there exists $x\in d_p(G)$ with $x\not\in t(G)$. Since $\Per(\mu_p)\subseteq t(G)$, Theorem \ref{s=0<->G=Per}(a) yields $s(\mu_p)>0$.
\end{proof}

\section{From one to infinitely many strings}

\subsection{Garlands and fans}

Our aim in this section is, starting from a string of a group endomorphisms $\f$, to construct infinitely many strings of $\f$. We propose two constructions.

\medskip
We leave without proof the following easy lemma.

\begin{lemma}
Let $G$ be an abelian group, $\f\in\End(G)$ and let $S=\{x_n\}_{n\in\N}$ be a pseudostring of $\f$ in $G$. For every $k\in\N_+$,
$$S_k:=\{x_n+x_{n+k}\}_{n\in\N}\ \text{and}\ S_k^*:=\{x_n+\ldots+x_{n+k}\}_{n\in\N}$$ are pseudostrings of $\f$. Furthermore, for every $a\in\N$ $$aS:=\{a x_n\}_{n\in\N}$$ is a pseudostring of $\f$. \hfill $\qed$
\end{lemma}

\begin{deff}
Let $G$ be an abelian group, $\f\in\End(G)$ and let $S=\{x_n\}_{n\in\N}$ be a pseudostring of $\f$ in $G$.
\begin{itemize}
\item[(a)] The \emph{garland} of $S$ is $\mathcal{G}(S)=\{S_k\}_{k\in\N_+}$.
\item[(b)] The \emph{convex garland} of $S$ is $\mathcal{G}^*(S)=\{S_k^*\}_{k\in\N}$.
\item[(c)] The $(a_k)$-\emph{fan} of $S$, with $\{a_k\}_{k\in\N}$ an infinite sequence of pairwise distinct integers, is $\mathcal{F}_{(a_k)}(S)=\{a_kS\}_{k\in\N}$.
\end{itemize}
If the elements of $\mathcal{G}(S)$ (respectively, $\mathcal{G}^*(S)$, $\mathcal{F}_{(a_k)}(S)$) are pairwise disjoint strings we say that the garland (respectively, the convex garland, the fan) is \emph{proper}. If the elements of $\mathcal G(S)$ are non-singular strings (respectively, null strings), we say that $\mathcal G(S)$ is a \emph{non-singular garland} (respectively, a \emph{null garland}); and analogously for $\mathcal G^*(S)$ and $\mathcal F_{(a_k)}(S)$.
\end{deff}

Let $S=\{x_n\}_{n\in\N}$ be a pseudostring of an endomorphism $\f$ of an abelian group $G$; moreover, consider $\mathcal G(S)=\{S_k\}_{k\in\N_+}$, $\mathcal G^*(S)=\{S^*_k\}_{k\in\N_+}$ and $\mathcal F_{(a_k)}(S)=\{a_k S\}_{k\in\N}$ for an infinite sequence of pairwise distinct integers $\{a_k\}_{k\in\N}$. If $H\leq G$ and $\pi:G\to G/H$ is the canonical projection, let $\pi(\mathcal{G}(S))=\{\pi(S_k)\}_{k\in\N_+}$ , $\pi(\mathcal{G}^*(S))=\{\pi(S_k^*)\}_{k\in\N_+}$ and $\pi(\mathcal F_{(a_k)}(S))=\{\pi(a_k S)\}_{k\in\N}$.

\begin{lemma}\label{lemmaX}
Let $G$ be an abelian group, $\f\in\End(G)$, $H$ a $\f$-invariant subgroup of $G$ and $\pi:G\to G/H$ the canonical projection. If $S=\{x_n\}_{n\in\N}$ is a pseudostring of $\f$ in $G$, and $\{a_k\}_{k\in\N}$ is a sequence of pairwise distinct integers, then:
\begin{itemize} 
\item[(a)] $\pi(\mathcal{G}(S))=\mathcal{G}(\pi(S))$,
\item[(b)] $\pi(\mathcal{G}^*(S))=\mathcal{G}^*(\pi(S))$, and
\item[(c)] $\pi(\mathcal{F}_{(a_k)}(S))=\mathcal{F}_{(a_k)}(\pi(S))$. 
\end{itemize}
Assume that $S$ is a string;
\begin{itemize}
\item[(a)] if $\mathcal{G}(\pi(S))$ is proper, then $\mathcal{G}(S)$ is proper;
\item[(b)] if $\mathcal{G}^*(\pi(S))$ is proper, then $\mathcal{G}^*(S)$ is proper;
\item[(c)] if $\mathcal{F}_{(a_k)}(\pi(S))$ is proper, then $\mathcal{F}_{(a_k)}(S)$ is proper.
\end{itemize}
\end{lemma}
\begin{proof}
The first part is clear, while the second part follows from Lemma \ref{quoziente}.
\end{proof}

\subsection{Relations among the string numbers, and proof of Theorem A$^{**}$}

The hyperkernel of a group endomorphism has the following easy to prove property, which in some sense permits to reduce the computation of the null string number to injective group endomorphisms.

\begin{lemma}\label{quoinj}
Let $G$ be an abelian group and $\f\in \End(G)$. Then the endomorphism $\overline\f :G/\ker_\infty\f\rightarrow G/\ker_\infty\f$ induced by $\f$ is injective. \hfill$\qed$
\end{lemma}

Moreover, one can easily prove that $\ker_\infty\f \leq Q\Per(\f)$, and more precisely that $Q\Per(\f) = \Per(\f ) + \ker_\infty\f$. 

\smallskip
The next lemma characterizes the null strings of a group endomorphism. The proof is clear.

\begin{lemma}\label{null<->iper}
Let $G$ be a an abelian group, $\f\in\End(G)$ and $S$ a string of $\f$. Then $S$ is a null string of $\f$ if and only if $S\subseteq \ker_\infty\f$. In particular, $s_0(\f)=s(\f\restriction_{\ker_\infty\f})=s_0(\f\restriction_{\ker_\infty\f})$. \hfill $\qed$
\end{lemma}

This lemma has an obvious consequence on the null string number of injective group endomorphisms:

\begin{corollary}\label{inj->s0=0}
Let $G$ be an abelian group and $\f\in\End(G)$. If $\f$ is injective, then $s_0(\f)=0$. \hfill $\qed$
\end{corollary}

The following lemma is easy to prove and gives a useful characterization of non-singular strings.

\begin{lemma}
Let $\Gamma$ be a set, $\lambda:\Gamma\to \Gamma$ a self-map and $S$ a string of $\lambda$. Then $S$ is non-singular if and only if $S\cap Q\Per(\lambda)= \emptyset$. \hfill$\qed$
\end{lemma}

This lemma has the following immediate consequence for group endomorphisms, since every null string is singular.

\begin{corollary}\label{non-sing:cap:null=emptyset}
Let $G$ be an abelian group and $\f\in\End(G)$. A non-singular string of $\f$ and a null string of $\f$ are always disjoint. \hfill$\qed$ 
\end{corollary}

Using only Corollary \ref{non-sing:cap:null=emptyset} it is possible to prove the following inequality involving the three string numbers. By making use of Proposition \ref{sing->null}, Theorem \ref{s=ns+s0} will show that in fact the equality holds.

\begin{theorem}\label{s>ns+s0}
Let $G$ be an abelian group and $\f\in \End(G)$. Then $s(\f) \geq ns(\f)+s_0(\f)$. 
\end{theorem}
\begin{proof}
If $s(\f)=0$ there is nothing to prove. The obvious inequalities $s(\f) \geq ns(\f)$ and $s(\f) \geq s_0(\f)$ settle the case when at least one of $ns(\f)$ and $s_0(\f)$ is either $0$ or $\infty$. Suppose that $0<ns(\f)<\infty$ and $0<s_0(\f)<\infty$; let ${\mathcal S}=\{S_i\}_{i\in I}$ and ${\mathcal S}'=\{S'_j\}_{j\in J}$ be respectively the families of pairwise disjoint non-singular strings and null strings witnessing this. By Corollary \ref{non-sing:cap:null=emptyset}, $S_i\cap S'_j=\emptyset$ for every $i\in I$ and $j\in J$. Therefore ${\mathcal S}\cup {\mathcal S}'$ is a family of pairwise  disjoint strings witnessing $s(\f) \geq ns(\f)+s_0(\f)$.
\end{proof}


\begin{proposition}\label{sing->null}
Let $G$ be an abelian group and $\f\in\End(G)$. If $\f$ admits a singular string, then $\f$ admits a null string. In particular, if $s(\f)>0$ and $ns(\f)=0$, then $s_0(\f)>0$.
\end{proposition}
\begin{proof}
Let $S=\{x_n\}_{n\in\N}$ be a singular string of $\f$ in $G$. Define $x_{-n}=\f^n(x_0)$ for every $n\in\N$. There exists a greatest $b\in\Z$ such that $x_b=x_a$ for some $a<b$ in $\Z$. Then $S'=\{y_n\}_{n\in\N}$, where $y_n=x_{a+n+1}-x_{b+n+1}$ for every $n\in\N$, is a null string. Indeed, $\f(y_n)=\f(x_{a+n+1})-\f(x_{b+n+1})=x_{a+n}-x_{b+n}=y_{n-1}$ for every $n\in\N_+$; moreover, if there exist $l<m$ in $\N$ such that $y_l=y_m$, then $x_{a+l+1}-x_{b+l+1}=x_{a+m+1}-x_{b+m+1}$. Applying $\f^{l+1}$ we have that $x_{a+m-l}=x_{b+m-l}$, in contradiction with the choice of $b$, since $m-l>0$. This shows that $S'$ is a string of $\f$. To conclude that $S'$ is a null string note that $y_0=x_{a+1}-x_{b+1}$ and $\f(y_0)=\f(x_{a+1}-x_{b+1})=x_a-x_b=0$.
\end{proof}


 
Example \ref{utilissimo} will show that the conjunction of $s(\f)=\infty$ and $ns(\f)=0$ may occur (in this case, $s_0(\f)=s(\f)=\infty$). In fact, in the hypothesis of the example, if $S=\{x_n\}_{n\in\N}$ is a string, then there exists $n\in\N_+$ such that $\f^n(x_0)=0$. Such a string is a null string, so in particular singular.

\medskip
The following proposition gives a sufficient condition for the existence of a null string.

\begin{proposition}\label{Hopfs}
Let $G$ be an abelian group and $\f\in\End(G)$. If $\f$ is non-injective and surjective, then $s_0(\f) > 0$.
\end{proposition}
\begin{proof} 
As $\f$ is non-injective, there exists $x_0\ne 0$ in $G$ such that $\f(x_0) =0$. Now define $S=\{x_n\}_{n\in\N}$ by induction, using the fact that $\f$ is surjective. To show that $S$ is a null string, it remains to prove that $S$ is a string. If $x_n=x_m$ for some $n\leq m$ in $\N$, then $x_0=x_{m-n}$ and so $\f(x_{m-n})=0$ implies $m=n$.
\end{proof}

Recall that a {\em Hopfian group} is a group for which every surjective endomorphism is an isomorphism. (Equivalently, a group is Hopfian if and only if it is not isomorphic to any of its proper quotients.)
Since Proposition \ref{Hopfs} implies that having null string number zero is a sufficient condition for a surjective group endomorphism to be an automorphism, we get the following

\begin{corollary}
Let $G$ be an abelian group. If $s_0(\f)=0$ for every $\f\in\End(G)$, then $G$ is Hopfian. \hfill$\qed$
\end{corollary}

It is worth asking whether the converse implication of this corollary holds. So we leave open the following problem.

\begin{question}
If $G$ is a Hopfian abelian group, is it true that every $\f\in\End(G)$ has $s_0(\f)=0$?
\end{question}

The next proposition shows in particular that $s_0(\f)>0$ implies $s_0(\f)=\infty$ for any endomorphism $\f$ of an abelian group $G$. This allows us to prove Theorem A$^{**}$.

\begin{proposition}\label{caso0}
Let $G$ be an abelian group and $\f\in\End(G)$. If $\f$ admits a null string $S=\{x_n\}_{n\in\N}$, then $\mathcal{G}(S)$ is proper and null.
\end{proposition}
\begin{proof}
Suppose without loss of generality that $\f(x_0)=0$.
Consider the garland $\mathcal G(S)=\{S_k\}_{k\in\N_+}$ of $S$.
We show first that $S_k$ is a string for every $k\in\N_+$. To this end let $k\in\N_+$ and by contradiction suppose that there exist $n<m$ in $\N$ such that $x_{n}+x_{n+k}=x_{m}+x_{m+k}$. From this relation we obtain, applying $\f^{m+k}$, that $x_{0}=0$, a contradiction.
We verify now that the $S_k$ are pairwise disjoint strings. By Lemma \ref{disgiunte}, it suffices to fix $k\in\N_+$ and to prove that $x_0+x_k\not\in S_j$ for every $i\neq k$ in $\N_+$. By contradiction suppose that there exists $n\in\N_+$ such that $x_0+x_k= x_{n}+x_{n+j}$. Since $x_0 \ne x_n$ as $n>0$, we have $k \ne n+j$. Let $i=\text{max}\{k,n+j\}$; applying $\f^i$, we obtain $x_0=0$, a contradiction.
\end{proof}

We can now prove the equality in Equation \eqref{s=ns+s0-eq}.

\begin{theorem}\label{s=ns+s0}
Let $G$ be an abelian group and let $\f\in\End(G)$. Then $s(\f)= ns(\f)+s_0(\f)$.
\end{theorem}
\begin{proof}
By Theorem \ref{s>ns+s0}, $s(\f) \geq ns(\f)+s_0(\f)$.
We have to prove that equality holds. If $s(\f)=0$ we are done. Assume that $s(\f)>0$. By Proposition \ref{sing->null}, either $ns(\f)>0$ or $s_0(\f)>0$. If $s_0(\f)>0$, then $s_0(\f)=\infty$ by Proposition \ref{caso0} and so we get the desired equality. So suppose that $s_0(\f)=0$. Let $S$ be a string of $\f$. If $S$ is singular, then Proposition \ref{sing->null} implies the existence of a null string, against our assumption. Then $S$ has to be non-singular. This shows that $s(\f)=ns(\f)$.
\end{proof}

Now we can prove Theorem A$^{**}$.

\begin{proof}[\bf Proof of Theorem A$^{**}$]
(i)$\Rightarrow$(ii) follows Proposition \ref{caso0}, and (ii)$\Rightarrow$(i) is trivial.

\smallskip
(i)$\Rightarrow$(iii) For each null string $S$ of $\f$ in $G$ the $\f$-invariant subgroup $H$ generated by $S$ (i.e., the subgroup of $G$ generated by $\bigcup_{n\in \N} \f^n(S)$) satisfies $\f(H)=H$, so $H\leq\sc(\f)$, and $H\leq\ker_\infty\f$. Then $\sc(\f)\cap\ker_\infty\f\neq0$.

\smallskip
(iii)$\Rightarrow$(i) Let $H=\sc(\f)\cap\ker_\infty\f\neq0$. Then $\f(H)=H$, because $x\in H$ implies that $x=\f(y)$ for some $y\in\sc(\f)$, and $y\in\ker_\infty\f$ since $x\in\ker_\infty \f$, so that $y\in H$. Take now any $x\in H\setminus\{0\}$ and build a pseudostring $S=\{x_n\}_{n\in\N}$ of $\f$ in $H$ with $x_0 =x$, using the fact that $\f(H)=H$. Then $S$ is a string since $x_0=x_n$ for some $n\in\N$ would imply $\f^k(x_0)\neq 0$ for every $k\in\N_+$, against the hypothesis that $x_0\in\ker_\infty\f$. Moreover, $S$ is a null string as $\f^k(x_0)=0$ for some $k\in\N_+$, since $x_0\in\ker_\infty\f$.
\end{proof}

\subsection{Examples}

One can see that no endomorphism of $\Z$ has strings:

\begin{example}
Let $\f\in\End(\Z)$.
Then $\f=\mu_a$ for some $a\in\Z$. 
Since for every $x\in\Z$, $x$ is not divisible by $a$ infinitely many times, every pseudostring of $\f$ is finite, so that it cannot be a string; hence $s(\f)=0$.
\end{example}

The situation is quite different if we consider non-cyclic free abelian groups:

\begin{example}\label{CorZan}
Let $G=\mathbb{Z}\oplus\mathbb{Z}$, and let $\f\in\text{End}(G)$ be defined by $\f(e_1)=(e_1)$ and $\f(e_2)=e_1+e_2$, where $e_1=(1,0)$ and $e_2=(0,1)$.
\begin{itemize} 
\item[(a)] Then $ns(\f) = \infty$ (so also $s(\f) = \infty$).

Indeed, let $S=\left\{(1,1),(0,1),(-1,1),\ldots,(-n,1),\ldots\right\}$. It is easy to see that $S$ is a non-singular string and that the fan $\mathcal F_{(k)}(S)$, for the sequence $\{k\}_{k\in\N}=\N$, is proper and non-singular.

\item[(b)] Moreover, the $\f$-invariant subgroup $H = \Per(\f)$ of $G$ satisfies $s(\f\restriction_H) = 0 = s(\overline \f)$ (so also $ns(\f\restriction_H)=0=ns(\overline\f)$). 

Indeed, $s(\f\restriction_H)=0$ by Proposition \ref{Per>NoStrings}(a). Moreover, $\Per(\f) =\langle e_1\rangle = \Z \oplus\{0\}$ and $G/\Per(\f)\cong\Z$. The endomorphism $\overline \f:G/\Per(\f)\to G/\Per(\f)$ induced by $\f$ is exactly $id_{G/\Per(\f)}$, because $\f(e_2)\in e_2 + \Per(\f)$; hence also $s(\overline\f)=0$. 
\item[(c)] Finally, since $\f$ is injective, $s_0(\f)=0$ by Corollary \ref{inj->s0=0}.
\end{itemize} 
\end{example}

Observe that (a) and (b) of Example \ref{CorZan} imply that the counterpart of the Addition Theorem (see Fact \ref{ent}(C)) for $s(-)$ and $ns(-)$ fails spectacularly.

\medskip
In the next example we see that every non-periodic endomorphism of $\mathbb Q$ has infinitely many (non-singular) strings.

\begin{example}\label{Q}
Let $G=\Q$ and $\f\in\End(G)\cong \Q$. Note that the only periodic endomorphisms of $\Q$ are $0_\Q$ and $\pm id_\Q$, which have no strings by Lemma \ref{q-p}. Assume that $\f\neq 0_\Q, id_\Q$. Then $s(\f) =ns(\f)= \infty$, while $s_0(\f)=0$ by Corollary \ref{inj->s0=0}. 

Indeed, $\f$ is the multiplication by some $r\in\mathbb{Q}\setminus\{0,\pm 1\}$. Consider the string $S=\left\{\frac{1}{r^n}\right\}_{n\in\N}$ and
let $\{p_k\}_{k\in\N}$ be a sequence of distinct primes coprime with the numerator and with the denominator of $r$. Then $\mathcal F_{(p_k)}(S)$ is proper and non-singular.
\end{example}

In the following example we see that a locally quasi-periodic endomorphism $\f$ may have strings, even if they have to be singular strings in view of Proposition \ref{Per>NoStrings}(b).

\begin{example}\label{utilissimo}
Let $p$ be a fixed prime and $\mathbb{Z}(p^{\infty})$ the Pr\"ufer group, with generators $c_n=\frac{1}{p^n}+\Z$, for $n\in \N_+$. Consider $\mu_p\in \End(\mathbb{Z}(p^{\infty}))$. 
\begin{itemize}
\item[(a)] Then $\mu_p$ is locally quasi-periodic.

Indeed, for $x\in\Z(p^\infty)$, there exists $n\in\N_+$ such that $o(x)=p^n$, and so $\mu_p^n(x)=0$.

\item[(b)] Moreover, $s_0(\mu_p)=\infty$, and so in particular $s(\mu_p)=\infty$.

Indeed, it is easy to see that $S=\{c_n\}_{n\in\N_+}$ is a null string of $\mu_p$. By Proposition \ref{caso0} the garland $\mathcal G(S)=\{S_k\}_{k\in\N_+}$ is proper and null.

\item[(c)]It is possible to verify that also $\mathcal G^*(S)$ is proper. On the other hand, for any sequence $\{a_k\}_{k\in\N}$ of pairwise distinct natural numbers, $\mathcal F_{(a_k)}(S)$ is not proper. 

Actually, $mS$ is a string only if $p^2\not | m$, while $mS$ and $kS$ are disjoint for $m,k \in \Z$ coprime to $p$ and such that $m\not \equiv k \mod p$.
\end{itemize}
\end{example}

Now one can prove the following result, which gives a condition equivalent to $s(\mu_p)=0$. It was inspired by Example \ref{sc-ex}(a), and a consequence is that the implications in Example \ref{sc-ex}(b,c) can be reverted, as shown by Corollary \ref{sc-rev}.

\begin{theorem}\label{s(mup)=0}
Let $G$ be an abelian group and $p$ a prime. Then $s(\mu_p)=0$ if and only if $d_p(G)\subseteq t(G)$ and $d_p(t_p(G))=0$.
\end{theorem}
\begin{proof}
If $d_p(G)\not\subseteq t(G)$, then $s(\mu_p)>0$ by Corollary \ref{dp-notin-t}. If $d_p(t_p(G))\not= 0$, then $t_p(G)$ contains a subgroup $H$ isomorphic to $\Z(p^\infty)$, and $s(\mu_p\restriction_H)>0$ by Example \ref{utilissimo}. By Lemma \ref{subgroups}, $s(\mu_p)>0$ as well.
If $d_p(G)\subseteq t(G)$, and $d_p(t_p(G))=0$, by Example \ref{sc-ex}(a), Lemma \ref{subgroups} and Corollary \ref{dptp}, $s(\mu_p)=s(\mu_p\restriction_{d_p(G)})=s(\mu_p\restriction_{t(G)})=0$.
\end{proof}

\begin{corollary}\label{sc-rev}
\begin{itemize}
\item[(a)] If $G$ is an abelian $p$-group, then $s(\mu_p)=0$ if and only if $G$ is reduced.
\item[(b)] If $G$ is a torsion-free abelian group, then $s(\mu_p) =0$ if and only if $p^\omega G = 0$. \hfill$\qed$
\end{itemize}
\end{corollary}

\begin{corollary}\label{s(mup)=0-cor}
Let $p$ be a prime, and let $G$ be a $p$-divisible abelian group. Then $s(\mu_p)=0$ if and only if $G$ is torsion and $t_p(G)=0$.
\end{corollary}
\begin{proof}
Assume that $s(\mu_p)=0$. By Theorem \ref{s(mup)=0}, this implies $d_p(G)\subseteq t(G)$ and $d_p(t_p(G))=0$.  Since $d_p(G)=G$, it follows that $G=t(G)=\bigoplus_{p}t_p(G)$, and also $G=d_p(G)=d_p(t_p(G))\oplus\bigoplus_{q\neq p}t_q(G)$. So $t_p(G)=d_p(t_p(G))=0$.
Suppose now that $G=t(G)$ and $t_p(G)=0$. By Theorem \ref{s(mup)=0}, $s(\mu_p)=0$.
\end{proof}

To complete the description of the values of the string numbers on $\mu_p$, the following lemma concerns torsion-free abelian groups.

\begin{lemma}\label{tfree->s=ns-mup}
If $G$ is a torsion-free abelian group, then $s_0(\mu_p)=0$ and $s(\mu_p)=ns(\mu_p)$.
\end{lemma}
\begin{proof}
If $\mu_p$ admits a null string $S=\{x_n\}_{n\in\N}$, then $S\subseteq t(G)=0$. This shows that $s_0(\mu_p)=0$. By Theorem \ref{s=ns+s0}, $s(\mu_p)=ns(\mu_p)$.
\end{proof}

\begin{example}\label{J-ex}
Let $p$ and $q$ be distinct primes. Corollary \ref{sc-rev}(b) implies $s(\mu_p)=0$ on $\J_p$. Moreover, $ns(\mu_p)=s(\mu_p)=\infty$ and $s_0(\mu_p)=0$ for $\J_q$ in view of Lemma \ref{tfree->s=ns-mup}, Corollary \ref{s(mup)=0-cor} and Corollary \ref{inj->s0=0}. Finally, $s(\mu_p)=0$ for $\Z(q^\infty)$ by Corollary \ref{s(mup)=0-cor}.
\end{example}

The first part of the next example shows that the monotonicity of the null string number under taking induced endomorphisms on quotients does not hold. So in particular the Addition Theorem fails for the null string number.

\begin{example}\label{s0-non-monotone}
Let $p$ be a prime and consider $\mu_p:\Q\to\Q$. Then $s_0(\mu_p)=0$ by Example \ref{Q}. Consider now $\overline\mu_p:\Q/\Z\to\Q/\Z$ induced  by $\mu_p$, which is still the multiplication by $p$. Since $\Q/\Z\geq\Z(p^\infty)$, and $s_0(\overline\mu_p\restriction_{\Z(p^\infty)})=\infty$ by Example \ref{utilissimo}, it follows that $s_0(\overline\mu_p)=\infty$ by Lemma \ref{subgroups}.

Moreover, for $\Q/\Z$, $s(\mu_p)=\infty$ for the same reason as above. Furthermore, $ns(\mu_p)=0$, because $\Q/\Z=\bigoplus_p\Z(p^\infty)$, and $ns(\mu_p)=0$ for $\Z(p^\infty)$ by Example \ref{utilissimo}, while $ns(\mu_p)=0$ on $\Z(q^\infty)$ for every prime $q$ different from $p$ in view of Example \ref{J-ex}.
\end{example}

We summarize in the following table the values, calculated in this section, of the string numbers of $\mu_p$, for different abelian groups and $p$ a prime. To have $\infty$ in all three columns, it suffices to take $\Q\oplus\Z(p^\infty)$, for $p$ a prime, and apply Lemma \ref{subgroups}.

\begin{center}
\begin{tabular}{|c|ccc|}
\hline
$\mu_p$ & $s(-)$ & $ns(-)$ & $s_0(-)$ \\
\hline
$\Z$ & $0$ & $0$ & $0$ \\ 
$\Q$ & $\infty$ & $\infty$ & $0$ \\
$\Z(p^\infty)$ & $\infty$ & $0$ & $\infty$ \\
$\Z(q^\infty)$ & $0$ & $0$ & $0$ \\
$\Q/\Z$ & $\infty$ & $0$ & $\infty$ \\
$\J_p$ & $0$ & $0$ & $0$ \\
$\J_q$ & $\infty$ & $\infty$ & $0$ \\
\hline
\end{tabular}
\captionof{table}{values on the multiplication by a prime $p$}\label{table-p}
\end{center}

In the following example we calculate the string number of the Bernoulli shifts, since they are fundamental examples in algebraic, but also topological and ergodic entropy theory.

\begin{example}\label{Example:shift}
Let $K$ be a non-trivial abelian group.
\begin{itemize}
\item[(a)]First we prove that $s(\beta_K^\oplus)=0$ (and so $ns(\beta_K^\oplus)=s_0(\beta_K^\oplus)=0$).

In fact, for every $x\in K^{(\N)}$, $x\neq 0$, there exists $m\in\mathbb N_+$ such that $(\beta_K^\oplus)^{-m}(x)$ is empty and so $\beta_K^\oplus$ cannot have any string in $K^{(\N)}$.

For an alternative proof, note that $\sc(\beta_K^\oplus)=0$ by Example \ref{sc-bernoulli}, and hence $s(\beta_K^\oplus)=0$ by Proposition \ref{s-c}.

\item[(b)]We verify now that $s({}_K\beta^\oplus)=s_0({}_K\beta^\oplus)=\infty$ and $ns({}_K\beta^\oplus)=0$. 

Since ${}_K\beta^\oplus$ is surjective and non-injective, by Proposition \ref{Hopfs} it admits a null string $S$ and by Proposition \ref{caso0} $\mathcal G(S)$ is proper and each $S_k\in\mathcal G(S)$ is a null string. Then $s({}_K\beta^\oplus)=s_0({}_K\beta^\oplus)=\infty$. Moreover, for each $x\in K^{(\N)}$ there exists $n\in\N_+$ such that $({}_K\beta^\oplus)^n(x)=0$ and so every string of ${}_K\beta^\oplus$ is a null string. In particular, $ns({}_K\beta^\oplus)=0$.


\item[(c)]Finally, $s({}^t\!\beta_K^\oplus)=ns({}^t\!\beta_K^\oplus)=\infty$ and $s_0({}^t\!\beta_K^\oplus)=0$.

The group $K^{(\Z)}$ can be written also as $K^{(\Z)}=\bigoplus_{i\in\Z}K_i$ with $K_i=K$ for all $i\in\Z$. Let $x$ be a non-zero element of $K$ and for each $i\in\Z$ let $g_i\in \bigoplus_{i\in\N} K_i$ be such that $\supp(g_i)=\{i\}$ and the $i$-eth entry of $g_i$ is exactly $x$. Then $\overline\beta_K^\oplus(g_i)=g_{i+1}$ for every $i\in\Z$. Then $S=\{g_{-i}\}_{i\in\N}$ is a string of ${}^t\!\beta_K^\oplus$ in $\bigoplus_{i\in\Z} K_i$.
It is clear that $\mathcal G^*(S)=\{S^*_k\}_{k\in\N_+}$ is proper, looking at the supports of the elements of each $S_k^*$. 
Hence $s({}^t\!\beta_K^\oplus)=\infty$. Moreover, these strings are non-singular and so $s({}^t\!\beta_K^\oplus)=ns({}^t\!\beta_K^\oplus)=\infty$.
By Corollary \ref{inj->s0=0}, $s_0({}^t\!\beta_K^\oplus)=0$.
\end{itemize}
\end{example}

We conclude this series of examples with the calculation of the string number of a generalized shift.

\begin{example}
Let $\Gamma$ be a set, and $\lambda:\Gamma\to\Gamma$ a self-map.
Following \cite{GB}, an \emph{infinite orbit} of $\lambda$ in $\Gamma$ is an infinite sequence of distinct elements $A=\{a_n\}_{n\in\N}$ such that $\lambda(a_n)=a_{n+1}$ for every $n\in\N$. Moreover, $$o(\lambda)=\sup\{|\mathcal F|: \mathcal F\ \text{is a family of pairwise disjoint infinite orbits of $\lambda$ in $\Gamma$}\}.$$

\medskip
Assume that $\lambda$ is surjective and $\lambda^{-1}(i)$ is finite for every $i\in\Gamma$. Let $K$ be a non-trivial finite abelian group and consider the generalized shift $\sigma_\lambda^\oplus:K^{(\Gamma)}\to K^{(\Gamma)}$. Then 
$$s(\sigma_\lambda^\oplus)=\begin{cases}0 & \text{if and only if}\ o(\lambda)=0,\\ \infty & \text{if and only if}\ o(\lambda)>0.\end{cases}$$

To verify this result, first assume that $o(\lambda)>0$. Then there exists an infinite orbit $A=\{a_n\}_{n\in\N}$ of $\lambda$ in $\Gamma$. Define $B=A\cup\bigcup_{n\in\N_+}\lambda^{-n}(a_0)$; then $B\supseteq \lambda(B)\cup\lambda^{-1}(B)$. By \cite[Proposition 6.2]{GB} $K^{(B)}$ is $\sigma_\lambda^\oplus$-invariant and $\sigma_\lambda^\oplus\restriction_{K^{(B)}}=\sigma_{\lambda\restriction_B}$.
By Lemma \ref{subgroups} $s(\sigma_\lambda^\oplus)\geq s(\sigma_\lambda^\oplus\restriction_{K^{(B)}})=s(\sigma_{\lambda\restriction_B}).$ So it suffices to prove that $s(\sigma_{\lambda\restriction_B})$ is infinite. To this end,
let $x$ be a non-zero element of $K$. For every $n\in\N$ let $x_n$ be the element of $K^{(\Gamma)}$ such that $\supp(x_n)=\{a_n\}$ and the $a_n$-th entry of $x_n$ is exactly $x$. In particular the $x_n$'s are pairwise distinct elements of $G$. Since $A$ is $\lambda$-invariant, it is possible to consider $\lambda\restriction_A:A\to A$, which is injective. Therefore $\sigma_{\lambda\restriction_A}(x_n)=x_{n-1}$ for every $n\in\N_+$. Consequently $S=\{x_n\}_{n\in\N}$ is a string of $\sigma_{\lambda\restriction_A}$. Then $\mathcal G^*(S)$ is proper: it suffices to look at the supports of the elements of the $S_k\in\mathcal G^*(S)$. In particular, $s(\lambda\restriction_A)=\infty$, and so by Lemma \ref{subgroups} also $s(\lambda)=\infty$.
 
\smallskip
Suppose now that $s(\sigma_\lambda^\oplus)>0$ and let $S=\{x_n\}_{n\in\N}$ be a string of $\sigma_\lambda^\oplus$. Let $F=\supp(x_0)$. By \cite[Claim 4.2(a)]{AADGH} $\supp((\sigma_\lambda^\oplus)^n(x_n))=\lambda^{-n}(\supp(x_n))$ for every $n\in\N$. But $(\sigma_\lambda^\oplus)^n(x_n)=x_0$ and $\lambda$ is surjective; then $\lambda^n(F)=\supp(x_n)$ for every $n\in\N$. If there exists $m\in\N_+$ such that $\lambda^m(F)\subseteq F\cup\ldots \cup\lambda^{m-1}(F)$, then $x_n\in K^{F\cup\ldots \cup\lambda^{m-1}(F)}$ for every $n\in\N$, and so $S\subseteq K^{F\cup\ldots \cup\lambda^{m-1}(F)}$ which is not possible because $S$ is infinite while $K^{F\cup\ldots \cup\lambda^{m-1}(F)}$ is finite.
Consequently $\lambda^n(F)\not\subseteq F\cup\ldots \cup\lambda^{n-1}(F)$ for every $n\in\N$. It follows that $L=\bigcup_{n=0}^\infty \lambda^n(F)$ is infinite. But $L=\bigcup_{x\in F}\bigcup_{n=0}^\infty\left\{x,\sigma_\lambda^\oplus(x),\ldots,(\sigma_\lambda^\oplus)^n(x)\right\}$, and $F$ is finite, hence there exists $x\in F$ such that $A=\bigcup_{n=0}^\infty\{x,\lambda(x),\ldots,\lambda^n(x)\}=O(x)$ is infinite; in particular the elements of $A$ are all distinct, so that $A$ is an infinite orbit of $\lambda$, that is, $o(\lambda)>0$.
\end{example}

\section{Theorems A and A$^*$}\label{proof}

\subsection{The case of free abelian groups}

We can generalize Example \ref{CorZan} in the following way:

\begin{theorem}\label{generalizzazione}
Let $\{p_n\}_{n\in\N}$ be the sequence of all primes in increasing order. Let $G$ be a torsion-free abelian group such that $\bigcap_{m\in\N} \bigcup_{n=m}^\infty p_n G =0$. If $\f\in\End(G)$ admits a string $S$, then there exists an infinite sequence of pairwise distinct primes $\{q_k\}_{k\in\N}$ such that $\mathcal F_{(q_k)}(S)$ is proper.
\end{theorem}
\begin{proof} 
First we show that the hypothesis implies that
\begin{itemize}
\item[($*$)] for every $x\in G\setminus\{0\}$, there exists an infinite sequence of pairwise distinct primes $\{q_k\}_{k\in\N}$, such that $q_l x \notin q_k G$ for all $l\ne k$ in $\N$.
\end{itemize}
Let $x \in G\setminus\{0\}$. Then $x\in G\setminus \bigcap_{m\in\N} \bigcup_{n=m}^\infty p_n G=\bigcup_{m\in\N} \bigcap_{n=m}^\infty (G\setminus p_n G)$, and so there exists $m\in\N$ such that $x\not\in p_n G$ for every $n\geq m$. For $k\in\N$ let $q_k=p_{k+m}$. Then for $l\neq k$ in $\N$ one has $q_l x \notin q_k G$; indeed, since $(q_l,q_k) = 1$, $q_m x \in q_k G$ would imply $x \in q_k G$. Then $\{q_k\}_{k\in\N}$ is the sequence required in ($*$).

Let $S=\{x_n\}_{n\in\N}$ be a string of $\f$, and let $\{q_k\}_{k\in\N}$ be the infinite sequence of pairwise distinct primes given by ($*$) for $x=x_0$. 
Consider $\mathcal F_{(q_k)}(S)=\{q_k S\}_{k\in\N}$. Since $G$ is torsion-free, each $q_k S$ is a string of $\f$. Moreover, the $S_k$'s are pairwise disjoint because $q_m x_0 \notin q_k G$ for all $m\ne k$ in $\N$ and hence Lemma \ref{disgiunte} applies. 
\end{proof}

The following corollary is a consequence of Theorem \ref{generalizzazione} and it shows in particular that an endomorphism $\f$ of a free abelian group admits a string precisely when it admits infinitely many strings, i.e., $s(\f)>0$ implies $s(\f)=\infty$.

\begin{corollary}\label{senza torsione}\label{free}
Let $G$ be a free abelian group and $\f\in\End(G)$. If there exists a string $S=\left\{x_n\right\}_{n\in\N}$ of $\f$, then there exists an infinite sequence of distinct primes $\{p_k\}_{k\in\N}$ such that $\mathcal F_{(p_k)}(S)$ is proper.
\end{corollary}
\begin{proof}
We verify that $G$ satisfes the hypothesis of Theorem \ref{generalizzazione}. Fix a basis $B$ of $G$ and $x\in G\setminus\{0\}$. Then  $x = \sum_i k_i b_i$, for some $b_i \in B$ and $k_i\in\Z\setminus\{0\}$. 
Hence, $x\in pG$  for some prime $p$ if and only if $p|k_i$ for all $i$. This proves that $x \in pG$ only for finitely many primes $p$, and so the hypothesis of Theorem \ref{generalizzazione} is satisfied. 
\end{proof}

\subsection{The string number of an injective endomorphism}

In this section we prove Theorem A for injective group endomorphisms in Proposition \ref{iniettiva}. We start with two technical results, which will be applied in its proof. 


\begin{lemma}\label{f.g.}
Let $G$ be an infinite finitely generated abelian group, $\f\in\End(G)$ injective, and $S=\{x_n\}_{n\in\N}$ a string of $\f$. Then there exists an infinite sequence $\{a_k\}_{k\in\N}$ of pairwise distinct natural numbers such that $\mathcal F_{(a_k)}(S)$ is proper.
\end{lemma}
\begin{proof}
Since $G$ is finitely generated, $G=F\oplus t(G)$ where $F$ is a free abelian group and $t(G)$ is finite. One can write in a unique way $x_n=t_n+c_n$ with $t_n\in t(G)$ and $c_n\in F$. Since $G$ is finitely generated, there exists $m\in\N_+$ such that $mt(G)=0$, and so $mG=mF$, which is a $\f$-invariant subgroup of $G$.
Then $mS=\{mx_n\}_{n\in\N}=\{mc_n\}_{n\in\N}$, because $mt_n=0$ for every $n\in\N$. We verify that $mS$ is a string. In fact, assume that $mS$ is not a string. By Lemma \ref{infinite->string} $mS$ is finite. Since $mc_a=mc_b$ implies $c_a=c_b$, because $F$ is free, also $\{c_n\}_{n\in\N}$ is finite. So we have the injection $S\hookrightarrow t(G)\times \{c_n\}_{n\in\N}$ of $S$ in a finite set, a contradiction. This proves that $mS$ is a string of $\f\restriction_{mG}$ in the free abelian group $mG$. Now apply Corollary \ref{free} to find an infinite sequence of pairwise distinct primes $\{q_k\}_{k\in\N}$ such that $\mathcal F_{(q_k)}(kS)$ is a proper fan of $\f\restriction_{mG}$. Let $a_k=q_k m$ for every $k\in\N_+$. Then $\mathcal F_{(a_k)}(S)=\mathcal F_{(q_k)}(mS)$, and so this is a proper fan of $\f$.
\end{proof}

The next proposition is a powerful tool which applies in the proof of Proposition \ref{iniettiva}.

\begin{proposition}\label{potente}
Let $G$ be an abelian group, $\f\in\End(G)$ injective, and $S=\{x_n\}_{n\in\N}$ a string of $\f$. If there exist $k<h$ in $\N_+$ such that 
\begin{equation}\label{perla}
x_h\in\pm x_0+\left\langle x_1,\ldots,x_k\right\rangle,
\end{equation}
then there exists an infinite sequence of pairwise distinct integers $\{a_k\}_{k\in\N}$ such that $\mathcal F_{(a_k)}(S)$ is proper.
\end{proposition}
\begin{proof}
We prove that the subgroup $H=\langle S\rangle$ is $\f$-invariant and finitely generated. Since $S$ is infinite, Lemma \ref{f.g.} can be applied to $H$ and $\f\restriction_H$ to conclude the proof.  

To prove that $H$ is $\f$-invariant it suffices to check that $\f(x_0)\in\langle S\rangle$. This follows immediately from that fact that $x_0=\pm x_h+a_1x_1+\ldots+a_kx_k$ for some integers $a_1,\ldots,a_k$, according to our hypothesis. Now we show that $S\subseteq \langle x_0,\ldots,x_h\rangle$, and this implies that $H=\langle S\rangle$ is finitely generated. To this end we prove by induction that $x_n\in\langle x_0,\ldots,x_h\rangle$ for every $n\geq h$. For $n=h$ this is obvious. Assume now that $n>h$ and $\f(x_n)=x_{n-1}\in\langle x_0,\ldots,x_h\rangle$. Then \eqref{perla} yields 
$$\f(x_n)\in\langle x_0,\ldots,x_{h-1}\rangle=\langle \f(x_1),\ldots,\f(x_h)\rangle=\f(\langle x_1,\ldots,x_h\rangle).$$ 
By the injectivity of $\f$ we conclude that $x_n\in\langle x_0,\ldots,x_h\rangle$. 
\end{proof}

The next proposition proves, as a byproduct, the equivalence between (i) and (ii) of Theorem A in case the group endomorphism is injective. In more details, given a string of an injective group endomorphism, it states that at least one between the garland or a fan is proper; in both cases, from one string it produces infinitely many pairwise disjoint strings.

\begin{proposition}\label{iniettiva}
Let $G$ be an infinite abelian group, $\f\in \End(G)$ injective, and $S=\left\{x_n\right\}_{n\in\N}$ a string of $\f$. Then
\begin{itemize}
\item[(i)] either $\mathcal{G}(S)$ is proper, or
\item[(ii)] there exists an infinite sequence $\{a_k\}_{k\in\N}$ of pairwise distinct natural numbers such that $\mathcal{F}_{(a_k)}(S)$ is proper.
\end{itemize}
\end{proposition}
\begin{proof}
Consider the garland $\mathcal G(S)=\{S_k\}_{k\in\N}$.

\smallskip
\textsc{Case 1.} Assume that $S_k$ is not a string for some $k\in\N$. Then $x_j+x_{j+k}=x_{j+a}+x_{j+a+k}$ for some $j\in \N$, $a\in\N_+$. Applying $\f^j$ we find $x_0+x_k=x_a+x_{a+k}$. In particular $x_{a+k}\in x_0+\langle x_1,\ldots, x_{a+k-1}\rangle$ and so Proposition \ref{potente} gives (ii).

\smallskip
\textsc{Case 2.} Suppose that there exist $l<m$ in $\N$ such that $S_l$ and $S_m$ have non-trivial intersection. As $x_l\ne x_m$, one has $x_0+x_l\ne x_0+x_m$. So we have two cases.

(a) If $x_0+x_l=x_j+x_{j+m}$ with $j>0$, then, since $m>l$, we get $x_{j+m}\in x_0+\langle x_1,\ldots, x_{j+m-1}\rangle$, and so Proposition \ref{potente} yields (ii).

(b) If $x_0+x_m=x_j+x_{j+l}$ with $j>0$, then $j+l\neq m$, otherwise $x_0=x_j$, a contradiction. 
If $j+l>m$, then $x_{j+l}\in x_0+\langle x_1,\ldots,x_{j+l-1}\rangle$;
if $j+l<m$, then $x_m\in -x_0+\langle x_1,\ldots,x_{m-1}\rangle$.
In both these cases Proposition \ref{potente} gives again (ii).

\smallskip
\textsc{Case 3.} If neither Case 1 nor Case 2 occur, then we have (i).
\end{proof}

Now we prove Theorem A and Theorem A$^*$.

\begin{proof}[\bf Proof of Theorem A$^*$]
(i)$\Rightarrow$(ii) Let $S=\{x_n\}_{n\in\N}$ be a non-singular string of $\f$ in $G$. Consider $\overline \f:G/\ker_\infty\f\to G/\ker_\infty\f$. By Lemma \ref{quoinj} $\overline\f$ is injective. So we can apply Proposition \ref{iniettiva}. If $\mathcal G(S+\ker_\infty\f)$ is proper then $\mathcal G(S)$ is proper too by Lemma \ref{lemmaX}(a). Otherwise there exists an infinite sequence of pairwise distinct integers $\{a_k\}_{k\in\N}$ such that $\mathcal F_{(a_k)}(S+\ker_\infty\f)$ is proper. Then $\mathcal F_{(a_k)}(S)$ is proper as well by Lemma \ref{lemmaX}(c).

\smallskip
(ii)$\Rightarrow$(i) is trivial and (i)$\Leftrightarrow$(iii) is Theorem \ref{s=0<->G=Per}(a).
\end{proof}

\begin{proof}[\bf Proof of Theorem A]
(i)$\Rightarrow$(ii) Since $s(\f)>0$, by Theorem \ref{s=ns+s0} either $ns(\f)>0$ or $s_0(\f)>0$. If $ns(\f)>0$, then $ns(\f)=\infty$ by Theorem A$^*$. If $s_0(\f)>0$, then $s_0(\f)=\infty$ by Proposition \ref{caso0}.

\smallskip
(ii)$\Rightarrow$(i) is trivial and (i)$\Leftrightarrow$(iii) is Theorem \ref{s=0<->G=Per}(b).
\end{proof}

\subsection{Consequences of Theorems A, A$^*$ and A$^{**}$}

The following corollary of Theorems A, A$^*$ and A$^{**}$ is a logarithmic law for the string numbers. This is a typical property of the algebraic entropy (see Fact \ref{ent}(B)).

\begin{corollary}\label{s>s*}
Let $G$ be an abelian group and $\f\in\End(G)$. Then $s(\f^k)=k s(\f)$, $ns(\f^k)=k ns(\f)$ and $s_0(\f^k)=k s_0(\f)$ for every $k\in\N_+$.\hfill$\qed$
\end{corollary}

The following is an immediate consequence of Lemma \ref{s=0<->s=0} and Theorems A, A$^*$ and A$^{**}$. It is a weak version of the Addition Theorem (see Fact \ref{ent}(C)), even if we saw that the Addition Theorem does not hold in full generality for the three string numbers.

\begin{corollary}
Let $G$ be an abelian group, $\f\in\End(G)$. Let $G_1,G_2$ be $\f$-invariant subgroups of $G$ such that $G=G_1\times G_2$, and let $\f_1=\f\restriction_{G_1}, \f_2=\f\restriction_{G_2}$. Then:
\begin{itemize}
\item[(a)] $s(\f)=s(\f_1)+s(\f_2)$;
\item[(b)] $ns(\f)=ns(\f_1)+ns(\f_2)$;
\item[(c)] $s_0(\f)=s_0(\f_1)+s_0(\f_2)$. \hfill$\qed$
\end{itemize}
\end{corollary}

As another consequence of Theorems A and A$^*$ we find the following relation between the string number and the non-singular string number.

\begin{corollary}
Let $G$ be an abelian group and $\f\in\End(G)$ surjective. Then $s(\f)=ns(\f)$ if and only if either $G\neq Q\Per(\f)$ or $G=\Per(\f)$.
\end{corollary}
\begin{proof}
If $G=Q\Per(\f)>\Per(\f)$, then $s(\f)=\infty$ and $ns(\f)=0$ by the Corollaries of Theorem A and Theorem A$^*$.
Assume now that $s(\f)>ns(\f)$; by Theorem A and A$^*$ it follows that $s(\f)=\infty$ and $ns(\f)=0$. By the corollaries of Theorem A and Theorem A$^*$, $G=Q\Per(\f)>\Per(\f)$.
\end{proof}

We have already seen that monotonicity under taking restrictions to invariant subgroups is always available for the string numbers. Example \ref{s0-non-monotone} shows that the null string number is not monotone under taking induced endomorphisms on quotients, even for surjective endomorphisms.
The next theorem shows that instead the string number and the non-singular string number are monotone under taking endomorphisms induced on a quotient by a surjective endomorphism.

\begin{theorem}\label{quotients}
Let $G$ be an abelian group, $\f\in\End(G)$ surjective, $H$ a $\f$-invariant subgroup of $G$ and $\overline\f:G/H\to G/H$ the endomorphism induced by $\f$. Then $s(\f)\geq s(\overline\f)$ and $ns(\f)\geq ns(\overline\f)$. 
\end{theorem}
\begin{proof}
By Theorems A and A$^*$ it suffices to prove that $s(\f)=0$ implies $s(\overline\f)=0$ and that $ns(\f)=0$ implies $ns(\overline \f)=0$. 
Since $\f$ is surjective, $\overline \f$ is surjective as well. By the Corollary of Theorem A, $s(\f)=0$ if and only if $G=\Per(\f)$. Then $G/H=\Per(\overline\f)$, that is, $s(\overline\f)=0$ by the same corollary. By the Corollary of Theorem A$^*$, $ns(\f)=0$ if and only if $G=Q\Per(\f)$. Then $G/H=Q\Per(\overline\f)$, that is, $ns(\overline\f)=0$ by the same corollary.
\end{proof}

Example \ref{Jp} shows that this monotonicity law is not satisfied by arbitrary group endomorphisms. 

\begin{example}\label{Jp}
Let $p$ be a prime and consider $\J_p$. Then $\Z$ is dense in $\J_p$, and $\J_p/\Z\cong\Q^{(\cont)}\oplus\bigoplus_{q\neq p}\Z(q^\infty)$. By Example \ref{J-ex} $s(\mu_p)=0$, while Theorem \ref{s(mup)=0} yields $s(\overline\mu_p)=\infty$, where $\overline\mu_p:\J_p/\Z\to \J_p/\Z$.

If now we consider the quotient $\Q^{(\cont)}$ of $\J_p/\Z$ (so it is also a quotient of $\J_p$), then the endomorphism $\overline\mu_p':\Q^{(\cont)}\to \Q^{(\cont)}$ induced by $\mu_p$ is still the multiplication by $p$. Therefore, $ns(\overline\mu_p')=s(\overline\mu_p')=\infty$, while $ns(\mu_p)=s(\mu_p)=0$ by Corollary \ref{sc-rev}(b). Note that $s_0(\overline\mu_p)=s_0(\mu_p)=0$ by Corollary \ref{inj->s0=0}.
\end{example}

Nevertheless, we leave open the following question.

\begin{problem}
Describe the endomorphisms $\f$ of abelian groups $G$ and the $\f$-invariant subgroups $H$ of $G$, such that $s(\f)\geq s(\overline\f)$ (respectively, $ns(\f)\geq ns(\overline\f)$), where $\overline\f:G/H\to G/H$ is the endomorphism induced by $\f$.
\end{problem}

\end{document}